\documentclass[11pt,a4paper]{article}

\usepackage{amsmath,amssymb,latexsym}
\usepackage{mathtools}
\usepackage{mathrsfs}
\usepackage{graphicx}
\usepackage{epsfig}
\usepackage{longtable}
\usepackage{amsthm}
\usepackage{xcolor}
\usepackage{hyperref}
\usepackage{tikz}
\usepackage[all]{xy}
\usepackage{booktabs}

\usepackage{lineno}

\newtheorem{theorem}{Theorem}
\newtheorem{lemma}[theorem]{Lemma}
\newtheorem{corollary}[theorem]{Corollary}
\newtheorem{proposition}[theorem]{Proposition}
\newtheorem{definition}{Definition}

\newtheorem{remark}{Remark}

\newcommand{\calC}{\mathcal{C}}
\newcommand{\IQ}{[0,1]^2}
\newcommand{\mup}{\mu_p}
\newcommand{\muinf}{\mu_\infty}
\newcommand{\sigp}{\sigma_p}

\newcommand{\abs}[1]{\left|#1\right|}

\begin{document}

\title{ \Large\bf Quantitative analysis of non-exchangeability in bivariate copulas: Sharp bounds, statistical tests and mixing constructions}
\author{Manuel Úbeda-Flores$^{\rm }$
\\
\small{$^{\rm }$Department of Mathematics, University of Almería, 04120 Almería, Spain}\\
\small{\texttt{mubeda@ual.es}}\\
}
\maketitle


\begin{abstract}
This paper studies the degree to which a bivariate copula fails to be symmetric under coordinate permutation, a property known as
non-exchangeability. Working within an axiomatic framework that quantifies this asymmetry through a family of $L^p$-based measures, we establish sharp bounds linking non-exchangeability to classical dependence and concordance measures, prove exact scaling laws under convex mixing that enable explicit construction of copulas with any prescribed degree of asymmetry, and characterise the class of maximally non-exchangeable copulas together with the feasible range of asymmetry--concordance pairs. On the inferential side, we propose a nonparametric permutation test for exchangeability with exact
finite-sample error control and consistency against all asymmetric alternatives, validated by Monte Carlo simulation and illustrated
on a real data set.
\end{abstract}
\bigskip

\noindent MSC2020: 60E05, 62H20.\medskip

\noindent {\it Keywords}: copula; non-exchangeability; measures of association; measures of asymmetry.

\section{Introduction}

Consider a pair of real-valued random variables $(X, Y)$ defined on a common probability space. Let $F_X$ and $F_Y$ denote their respective marginal distribution functions (d.f.'s), and let $F_{X,Y}$ and $F_{Y,X}$ represent the joint d.f.'s of the vectors $(X, Y)$ and $(Y, X)$. Following \cite{Chow78}, the random vector $(X, Y)$ is classified as \emph{exchangeable} provided that $F_X = F_Y$ and $F_{X,Y} = F_{Y,X}$.

As pointed out by Nelsen \cite{Nel07}, while the concept of exchangeability has been extensively explored in the statistical literature, there has been a notable lack of research regarding the specific mechanisms through which identically distributed variables deviate from this property, as well as the underlying dependence structures that arise in such cases. However, several studies \cite{AP07,Beliakov2014,Dur09,Erd06,KlMes06,Kokol2020,Papini2015}  investigate the non-exchangeability of bivariate distributions, particularly focusing on their implementation within various statistical frameworks.

Inspired by these developments, the present work aims to provide an axiomatic foundation for quantifying the degree of non-exchangeability in bivariate vectors composed of continuous and identically distributed random variables. Special emphasis is placed on the significance and application of these measures within the context of copula theory.

The paper is organised as follows. Section~\ref{sec:prel} reviews the necessary
background on copulas and the axiomatic construction of non-exchangeability measures. Sections~\ref{sec:bounds} and~\ref{sec:monoton} establish sharp bounds relating non-exchangeability to classical dependence and concordance measures, and study how the measures vary with the integrability parameter. Section~\ref{sec:M1} examines the class of maximally non-exchangeable copulas and characterises the pairs of asymmetry and concordance that can be simultaneously achieved. Section~\ref{sec:escalado} proves exact scaling laws for the non-exchangeability measures under convex mixing, enabling explicit prescription of any target degree of asymmetry. Section~\ref{sec:test} develops the permutation test, its asymptotic theory, and its comparison with existing procedures. Section~\ref{sec:simulation} presents the Monte
Carlo study. Finally, Section~\ref{sec:applications} is devoted to applications and Section~\ref{sec:conc} collects the conclusions and open problems.

\section{Preliminaries}\label{sec:prel}

In order to establish a rigorous framework for our analysis, this section provides the necessary mathematical background that underpins the rest of this study. We begin by reviewing the fundamental properties of copulas, which serve as the essential tools for modeling the dependence structure of multivariate distributions independently of their marginals. Building upon these definitions, we recall the axiomatic construction of non-exchangeability measures.

\subsection{Basic definitions and properties of copulas}

We first review several fundamental concepts regarding copulas. For a complete survey, see \cite{Durante2016,Nel06}.

Formally, a (bivariate) \emph{copula} is a function $C: [0, 1]^2 \to [0, 1]$ that satisfies the following two conditions:
\begin{itemize}
\item[(C1)] \emph{Boundary constraints}: For every $t \in [0, 1]$, $C(t, 0) = C(0, t) = 0$ and $C(t, 1) = C(1, t) = t$.
\item[(C2)] \emph{The $2$-increasing property}: For any $u, v, u', v' \in [0, 1]$ such that $u \le u'$ and $v \le v'$, the $C$-{\it volume} of the rectangle $[u, u'] \times [v, v']$ is non-negative, i.e.,
\begin{equation*}
C(u', v') - C(u, v') - C(u', v) + C(u, v) \geq 0.
\end{equation*}
\end{itemize}

According to \emph{Sklar's theorem} \cite{Sk59}, for any pair of continuous random variables $(X, Y)$, their joint d.f. $F_{X,Y}$ can be uniquely represented for all $x, y \in [-\infty,+\infty]$ as
\begin{equation*}\label{E:sklar_new}
F_{X,Y}(x, y) = C(F_X(x), F_Y(y)),
\end{equation*}
where $C$ is a copula (for a complete proof of this result, see \cite{Ubeda2017}). Essentially, $C$ corresponds to a bivariate distribution function on $[0, 1]^2$ with uniform marginals.

For any copula $C \in \mathcal{C}$ ---where $\mathcal{C}$ denotes the set of all copulas---, the following inequality holds:
$W(u, v) \le C(u, v) \le M(u, v)$ for all $(u,v)\in[0,1]^2$, where $W(u, v) = \max(u+v-1, 0)$ and $M(u, v) = \min(u, v)$ are the \emph{Fr\'echet--Hoeffding lower and upper bounds}, representing perfect negative and positive dependence (countermonotonicity and comonotonicity), respectively. Furthermore, the product copula $\Pi(u, v) = uv$ characterizes the case where $X$ and $Y$ are independent.

The exchangeability of a random pair can be elegantly characterized through its associated copula as follows: Let $(X, Y)$ be a vector of continuous random variables with copula $C$. Then $X$ and $Y$ are exchangeable if, and only if, they are identically distributed ($F_X = F_Y$) and $C$ is \emph{symmetric}, i.e., $C(u, v) = C(v, u)$ for all $(u, v) \in [0, 1]^2$.


For any $C \in \mathcal{C}$, we define its transpose as $C^t(u, v) = C(v, u)$; and, additionally, the \emph{survival copula} $\hat{C}$ is given by $\hat{C}(u, v) = u + v - 1 + C(1-u, 1-v)$ for all $(u,v)\in[0,1]^2$. The following properties describe how transformations of the random variables affect the underlying copula: Let $X$ and $Y$ be continuous random variables with marginals $F_X, F_Y$ and copula $C$. Then:
\begin{enumerate}
\item[(a)] The copula representing the vector $(Y, X)$ is the transpose copula $C^t$.
\item[(b)] If $g: \mathbb{R} \to \mathbb{R}$ is a strictly increasing transformation, the copula of $(g(X), g(Y))$ remains $C$.
\item[(c)] If $g: \mathbb{R} \to \mathbb{R}$ is a strictly decreasing transformation, the copula of $(g(X), g(Y))$ becomes the survival copula $\hat{C}$.
\end{enumerate}

\subsection{Measures of non-exchangeability}

The \emph{degree of non-exchangeability} of a copula $C$ is quantified by
\[
  \mup(C) \;:=\; d_p(C,C^t), \qquad p\in[1,+\infty],
\]
where $d_p$ is the $L^p$ distance on $\calC$.
Explicitly, for $p<\infty$,
\[
  \mup(C)
  =\left(\int_{\IQ}\bigl|C(u,v)-C(v,u)\bigr|^p\,du\,dv\right)^{1/p},
\]
while $$\muinf(C)=\sup_{(u,v)\in\IQ}|C(u,v)-C(v,u)|.$$
It is proved in~\cite{Durante2010} that $\mup$ satisfies a natural set of axioms (boundedness; nullity if, and only if, $C$ is symmetric;
invariance under transposition and under strictly monotone transformations;
and continuity with respect to weak convergence), and that the maximum
value is $K_{\mu_\infty}=\tfrac{1}{3}$ and
$$K_{\mu_p}=\left(\frac{2\cdot 3^{-p}}{(p+1)(p+2)}\right)^{1/p}\quad\text{for}\,\, p<\infty.$$
The measure $\mu_p$ takes values in $[0, K_{\mu_p}]$. One could normalise $\mu_p$ by dividing by $K_{\mu_p}$ to obtain a measure in $[0,1]$; we do not adopt this convention here, and all subsequent results use the unnormalised form.

\section{Upper bounds for \texorpdfstring{$\mu_p$}{} in terms of known measures}\label{sec:bounds}

One of the most suggestive observations in \cite [Remark~1]{Durante2010} is that the non-exchangeability measure $\mu_p(C)$ is bounded above by the Schweizer and Wolff dependence measure $\sigma_p(C)$, given by
$$\sigma_p(C):=c_p\cdot d_p(C,\Pi)=\left(\frac{(2p+2)!}{2\,(p!)^2}\cdot\int_{\IQ}\bigl|C(u,v)-uv\bigr|^p\,{\rm d}u\,{\rm d}v\right)^{1/p}$$
for $p\in[1,\infty)$ and 
$$\sigma_{\infty}(C)
  :=c_\infty\cdot d_\infty(C,\Pi)=
  4\sup_{(u,v)\,\in\,\IQ}\abs{C(u,v)-uv}$$
(see \cite{Sch1981}), although the optimal constant was not determined, nor was the result extended to other classical measures. The objective of this section is precisely to address this gap and to establish explicit and optimal upper bounds for $\mu_p(C)$ as a function of Spearman's $\rho$, given by
\begin{align*}
\rho(C) &= 12 \int_{[0,1]^2} C(u,v) \, {\rm d}u\, {\rm d}v - 3
\end{align*}
(see \cite{Nel06}). The key insight is that the (concordance) measures quantify the global proximity of $C$ to the extreme copulas $M$ and $W$, whereas $\mu_p(C)$ measures the asymmetry of $C$ with respect to the diagonal. These two notions are not independent: a copula close to $M$ or $W$ has little margin to be asymmetric, as these copulas are themselves symmetric. Precisely formalizing this intuitive principle is the common thread of this section.

\subsection{Bound via the Schweizer and Wolff dependence measure}
\label{subsec:bound_sigma}

We begin by providing a precise formulation ---with an optimal constant--- of the bound mentioned without detail in the reference article. The primary tool is the triangle inequality in the metric space $(\mathcal{C}, d_p)$.

\begin{proposition}\label{prop:bound}
  Let $C\in\calC$. Then
  \begin{equation}\label{eq:bound}
    \mup(C) \;\le\; \frac{2}{c_p}\,\sigp(C),
  \end{equation}
for all $p\in[1,+\infty]$. Moreover, the constant $2/c_p$ is sharp for every $p\in[1,+\infty]$.
\end{proposition}
 
\begin{proof}
The change of variables $(u,v)\mapsto(v,u)$ preserves the Lebesgue measure on $\IQ$ and the supremum over $\IQ$, so the map $C\mapsto C^t$
is an isometry of $(\calC,d_p)$ for every $p\in[1,+\infty]$. In particular, $\Pi^t=\Pi$ gives $d_p(\Pi,C^t)=d_p(\Pi,C)$, and the triangle inequality in $(\calC,d_p)$ yields
\[
  \mup(C)
  = d_p(C,C^t)
  \;\le\;
  d_p(C,\Pi)+d_p(\Pi,C^t)
  = 2\,d_p(C,\Pi)
  = \frac{2}{c_p}\,\sigp(C).
\]

Now, let $B \in \mathcal{C}$ be any absolutely continuous copula with bounded density $b = \frac{\partial^2}{\partial u\,\partial v}B$ and
$B \neq B^t$, and define, for $\varepsilon > 0$\,:
\begin{equation}\label{eq:Ceps}
  C_\varepsilon(u,v) := uv + \frac{\varepsilon}{2}[B(u,v) - B(v,u)].
\end{equation}
Since $b$ is bounded, $C_\varepsilon$ has density
$$1 + \frac{\varepsilon}{2}[b(u,v) - b(v,u)]\ge 0$$ for all
$\varepsilon \le 2/\|b - b^t\|_\infty$, so $C_\varepsilon$ is a copula. Such copulas $B$ are dense in $(\mathcal{C}, d_p)$, so the sharpness
extends to all of $\mathcal{C}$.
 
By construction, $C_\varepsilon - \Pi = (\varepsilon/2)(B-B^t)$ is antisymmetric, so $C_\varepsilon(u,v)+C_\varepsilon(v,u)=2uv$
for all $(u,v)\in[0,1]^2$. Therefore
$C_\varepsilon(u,v)-C_\varepsilon(v,u) = \varepsilon(B(u,v)-B(v,u))$,
and
\begin{align}
  \mup(C_\varepsilon)
  &= d_p(C_\varepsilon,\,C_\varepsilon^t)
   = \varepsilon\,d_p(B,B^t)
   = \varepsilon\,\mup(B),
  \label{eq:mu_eps}\\[4pt]
  d_p(C_\varepsilon,\Pi)
  &= \frac{\varepsilon}{2}\,d_p(B,B^t)
   = \frac{\varepsilon}{2}\,\mup(B),
  \label{eq:d_eps}
\end{align}
where~\eqref{eq:d_eps} uses the fact that $C_\varepsilon-\Pi$
is antisymmetric and hence
$d_p(C_\varepsilon,\Pi)^p = (\varepsilon/2)^p\,d_p(B,B^t)^p$.
From~\eqref{eq:mu_eps} and~\eqref{eq:d_eps},
\[
  \frac{\mup(C_\varepsilon)}{\sigp(C_\varepsilon)}
  = \frac{\varepsilon\,\mup(B)}{c_p\cdot(\varepsilon/2)\,\mup(B)}
  = \frac{2}{c_p}
\]
for every $\varepsilon>0$. The bound~\eqref{eq:bound} is therefore sharp.
\end{proof}
 
\begin{remark}
The copula $C_\varepsilon$ defined in~\eqref{eq:Ceps} satisfies
$\mup(C_\varepsilon)/\sigp(C_\varepsilon)=2/c_p$ \emph{exactly},
for every $\varepsilon>0$ and every $B\ne B^t$.
The key property that makes the triangle inequality an equality in $d_p$
is that $C_\varepsilon-\Pi$ is antisymmetric, which forces $\Pi$ to lie
exactly on the $d_p$-geodesic between
$C_\varepsilon$ and $C_\varepsilon^t$.
\end{remark}

\subsection{Bound via Spearman's \texorpdfstring{$\rho$}{}}
\label{subsec:bound_rho}

Spearman's $\rho$ is related to $\sigma_1$. This allows the previous bound to be converted into a direct expression in terms of $\rho$, with the additional advantage that $\rho$ is easily estimable from data.

\begin{proposition}\label{teo32}
For every $C\in\calC$,
\begin{equation}\label{eq:chain}
  \mu_1(C)
  \;\le\;
  \frac{\sigma_1(C)}{6}
  \;\le\;
  \frac{2-\abs{\rho(C)}}{6}.
\end{equation}
\end{proposition}
 
\begin{proof}
The left inequality follows from Proposition \ref{prop:bound} with $p=1$.

For the second inequality, write $f(u,v)=C(u,v)-uv$ and split the integral by sign:
\begin{align*}
  \frac{\sigma_1(C)}{12}
  &=\int_{\{f(u,v)\ge 0\}}f(u,v)\,{\rm d}u\,{\rm d}v+\int_{\{f(u,v)<0\}}\abs{f(u,v)}\,{\rm d}u\,{\rm d}v,\\
  \frac{\rho(C)}{12}
  &=\int_{\{f(u,v)\ge 0\}}f(u,v)\,{\rm d}u\,{\rm d}v-\int_{\{f(u,v)<0\}}\abs{f(u,v)}\,{\rm d}u\,{\rm d}v.
\end{align*}
Adding and subtracting, the positive part of $f$ satisfies
$$\int_{\{f(u,v)\ge 0\}}f(u,v)\,{\rm d}u\,{\rm d}v=\frac{\sigma_1(C)+\rho(C)}{24}$$
and is bounded above by $$\int_{\IQ}[M(u,v)-\Pi(u,v)]\,{\rm d}u\,{\rm d}v=\frac{1}{12}.$$ Similarly, the negative part satisfies
$$\int_{\{f(u,v)<0\}}\abs{f(u,v)}\,{\rm d}u\,{\rm d}v=\frac{\sigma_1(C)-\rho(C)}{24}$$
and is bounded above by $$\int_{\IQ}[\Pi(u,v)-W(u,v)]\,{\rm d}u\,{\rm d}v=\frac{1}{12}.$$
Hence:
\[
  \frac{\sigma_1(C)+\rho(C)}{24}\le\frac{1}{12},
  \qquad
  \frac{\sigma_1(C)-\rho(C)}{24}\le\frac{1}{12},
\]
which give $\sigma_1(C)\le 2-\rho(C)$ and $\sigma_1(C)\le 2+\rho(C)$
respectively. Combining:
\[
  \sigma_1(C)\le\min\{2-\rho(C),\,2+\rho(C)\}=2-\abs{\rho(C)},
\]
whence the result follows.
\end{proof}

\begin{remark}
Proposition~\ref{teo32} admits a twofold interpretation.

(i) The closer $|\rho(C)|$ is to $1$, the smaller the upper bound $(2-|\rho(C)|)/6$ on $\mu_1(C)$ becomes. In the extreme case $|\rho(C)|=1$, the bound yields $\mu_1(C) \le 1/6$, which is
not tight enough to force $\mu_1(C)=0$ directly. However, one can argue independently: $\rho(C)=1$ implies $C=M$ and $\rho(C)=-1$ implies $C=W$; both $M(u,v)=\min(u,v)$ and $W(u,v)=\max(u+v-1,0)$ are symmetric ($M=M^t$, $W=W^t$), so $\mu_1(C)=0$ follows directly from the definition $\mu_1(C)=d_1(C,C^t)=0$.

(ii) The bound
\[
    \mu_1(C) \;\le\; \frac{2-|\rho(C)|}{6}
\]
provides a necessary condition on the pair $(\mu_1(C),\rho(C))$ that every copula must satisfy. In practice, if empirical estimates
$\hat{\mu}_n$ and $\hat{\rho}_n$ violate this inequality, i.e.,
\[
    \hat{\mu}_n \;>\; \frac{2-|\hat{\rho}_n|}{6},
\]
this signals sampling error or model misspecification. We note, however, that since $(2-|\rho|)/6 \ge 1/6 > K_{\mu_1}$ for all
$|\rho|\le 2/3$, the bound only becomes an active constraint for $|\rho|>2/3$, i.e., when concordance is strong. For moderate or weak concordance, the universal sharp bound $\mu_1(C)\le K_{\mu_1}$ is tighter and should be used as the diagnostic threshold instead.
\end{remark}

\section{Monotonicity of \texorpdfstring{$\mu_p$}{} in the parameter \texorpdfstring{$p$}{}}\label{sec:monoton}

The family of measures $\{\mu_p\}_{p\in[1,+\infty]}$ is not constant in $p$: 
the larger $p$ is, the more weight the $L^p$ norm assigns to the extreme values of the antisymmetric part $A(u,v)=C(u,v)-C(v,u)$. The following result formalizes this intuition precisely; however, to achieve this, we first require a well-known inequality whose proof can be found, for example, in \cite{Rudin1979}.

\begin{lemma}[Jensen's Inequality]
Let $(\Omega, \mathcal{A}, \mu)$ be a measure space such that $\mu(\Omega) = 1$. If $g$ is a real $\mu$-integrable function and $\varphi$ is a convex function on the real line, then:
\[
\varphi \left( \int_{\Omega} g \, d\mu \right) \leq \int_{\Omega} \varphi \circ g \, d\mu.
\]
\end{lemma}

\begin{proposition}\label{prop:monotp}
  For every copula $C\in\mathcal{C}$ and any $1\le p_1\le p_2\le+\infty$, it holds that:
  \[
    \mu_{p_1}(C) \;\le\; \mu_{p_2}(C).
  \]
  In particular, $\mu_1(C)\le\mu_p(C)\le\mu_{\infty}(C)$ for all $p\in[1,+\infty]$.
\end{proposition}

\begin{proof}
Let $A(u,v):=C(u,v)-C(v,u)$ be the antisymmetric part of $C$. We consider two cases:

\textit{Case $1\le p_1\le p_2<+\infty$}.
The ratio $r:=p_2/p_1\ge 1$ satisfies $r>0$, so the function $\varphi(t):=t^r$ is convex on $[0,+\infty)$. Applying Jensen's inequality to the convex function $\varphi$ with respect to the probability measure $\lambda^2$ (Lebesgue measure on $\mathbf{I}^2$ with total mass $1$):
\[
  \varphi\!\left(\int_{\mathbf{I}^2}\abs{A}^{p_1}\,du\,dv\right)
  \;\le\;
  \int_{\mathbf{I}^2}\varphi\!\left(\abs{A}^{p_1}\right)du\,dv,
\]
that is,
\[
\left(\int_{\mathbf{I}^2}\abs{A}^{p_1}\,du\,dv\right)^{\!r}
  \;\le\;
  \int_{\mathbf{I}^2}\abs{A}^{p_1 r}\,du\,dv
  = \int_{\mathbf{I}^2}\abs{A}^{p_2}\,du\,dv.
\]
Taking the $1/p_2$ power on both sides and using the fact that $r/p_2 = (p_2/p_1)/p_2 = 1/p_1$:
\[
  \left(\int_{\mathbf{I}^2}\abs{A}^{p_1}\right)^{\!1/p_1}
  \;\le\;
  \left(\int_{\mathbf{I}^2}\abs{A}^{p_2}\right)^{\!1/p_2},
\]
which means $\mu_{p_1}(C)\le\mu_{p_2}(C)$.

\textit{Case $p_2=+\infty$}. For any $p\in[1,+\infty)$, since $\abs{A(u,v)}\le\mu_{\infty}(C)$ for all $(u,v)\in\mathbf{I}^2$, we have $\abs{A(u,v)}^p\le\mu_{\infty}(C)^p$, and therefore:
\[
  \mu_p(C)^p
  = \int_{\mathbf{I}^2}\abs{A(u,v)}^p\,du\,dv
  \;\le\;
  \mu_{\infty}(C)^p\int_{\mathbf{I}^2}du\,dv
  = \mu_{\infty}(C)^p,
\]
whence $\mu_p(C)\le\mu_{\infty}(C)$.
\end{proof}

\begin{remark}
We wish to point out that the proof of Proposition \ref{prop:monotp} essentially relies on the fact that $\mathbf{I}^2$ has total measure $1$. In a measure space where $m \ne 1$, Jensen's inequality would yield an additional factor $m^{r-1}$ that could reverse the direction of the inequality; it is precisely the normalization of the Lebesgue measure on $\mathbf{I}^2$ that guarantees the monotonicity in the correct direction.

Note also that the case $p_2=+\infty$ is not a direct consequence of the case $p_1\le p_2<+\infty$ (it is not sufficient to simply take the limit $p_2\to\infty$, as the convergence $\mu_{p_2}\to\mu_{\infty}$ would need to be justified); instead, it is proven directly by bounding the integrand pointwise.
\end{remark}

\section{Maximally non-exchangeable copulas}
\label{sec:M1}
 
Let $M_1$ be the set
\[
  M_1 := \left\{C\in\calC : \muinf(C)=\frac{1}{3}\right\},
\]
the class of \emph{maximally non-exchangeable} copulas with respect to $\muinf$, since $K_{\muinf}=1/3$ is the supremum of $\muinf$ over $\calC$ \cite{Durante2010}.

A key element of this class is the copula $M_\theta$, given by
\[
  M_\theta(u,v) = \min\bigl\{u,\,v,\,(u-1+\theta)^++(v-\theta)^+\bigr\},
  \quad \theta\in[0,1/3],
\]
for all $(u,v)\in[0,1]^2$, where $(x)^+=\max\{0,x\}$ (see \cite{Durante2010}). For this copula we have $\rho(M_\theta)=1-6\theta+6\theta^2$ and $\tau(M_\theta)=(1-2\theta)^2$,
where $\tau$ represents the Kendall's $\tau$ measure of concordance, given by
\begin{align*}
\tau(C) &= 4 \int_{[0,1]^2} C(u,v) \, {\rm d}C(u,v) - 1
\end{align*}
(see \cite{Nel06}). In particular, $\rho(M_{1/3}) = -1/3$, 
  $\tau(M_{1/3}) = 1/9$, and $\muinf(M_{1/3}) = 1/3$. Note that $M_{1/3}\in M_1$ and both concordance measures are non-zero. This shows that the class $M_1$ contains copulas with non-trivial
concordance, and that the relationship between maximal non-exchangeability
and concordance is more subtle than a simple zero-concordance statement.
 
Theorem~2 of \cite{Durante2010} guarantees that for each $\theta\in[0,1]$ there exists a copula with $\muinf$-value exactly $\theta/3$: it suffices to take $C_\theta=\theta C_1+(1-\theta)C_s$, where $C_1\in M_1$ and $C_s$ is any symmetric copula ($C_s^t=C_s$, so $\muinf(C_s)=0$). We also have
$$\rho(C_\theta)= \theta\,\rho(C_1)+(1-\theta)\,\rho(C_s).$$
The identity $\mu_\infty(C_\theta) = \theta\,\mu_\infty(C_1)$ follows from Proposition~\ref{prop:mezclaS} in Section~\ref{sec:escalado}, which shows that mixing with a symmetric copula scales non-exchangeability linearly; we apply it here anticipating that result. In particular, as $\theta$ grows from $0$ to $1$, the non-exchangeability $\muinf(C_\theta)$ increases monotonically from $0$ to $1/3$, while $\rho(C_\theta)$ interpolates linearly between its values at $C_s$ and $C_1$. Given target values $\mu_0\in[0,1/3]$ and $r_0\in[-1,1]$, this allows us to construct a copula with $\muinf=\mu_0$ and $\rho=r_0$ provided there exist $C_1\in M_1$ and $C_s$ symmetric with
\[
  \theta\,\rho(C_1) + (1-\theta)\,\rho(C_s) = r_0,
  \qquad \theta = 3\mu_0.
\]
This is a linear equation in $\rho(C_s)$ (once $C_1$ is fixed), solvable whenever $r_0$ lies in the range of the right-hand side over all symmetric $C_s$. Since $\rho(C_s)$ can be any value in $[-1,1]$, the solvable range of $r_0$ is the interval
$[\theta\rho(C_1)-(1-\theta),\;\theta\rho(C_1)+(1-\theta)]$. For $\theta=3\mu_0$ and $C_1=M_{1/3}$ (with $\rho(C_1)=-1/3$):
\[
  r_0 \;\in\;
  \bigl[-3\mu_0\cdot\tfrac{1}{3}-(1-3\mu_0),\;
         -3\mu_0\cdot\tfrac{1}{3}+(1-3\mu_0)\bigr]
  = \bigl[2\mu_0-1,\; 1-4\mu_0\bigr].
\]
This characterises the feasible range of $\rho$ for given $\mu_0$ when $C_1=M_{1/3}$, and can be used as a diagnostic in model fitting.

\section{Scaling of \texorpdfstring{$\mup$}{} under mixing}\label{sec:escalado}

In Copula Theory, \emph{mixing} a copula with its transpose has a natural statistical interpretation: if $(X,Y)$ has copula $C$ and $(Y,X)$ has copula $C^t$, then $C_\alpha=\alpha C+(1-\alpha)C^t$
is the copula of a vector that behaves like $(X,Y)$ with probability
$\alpha$ and like $(Y,X)$ with probability $1-\alpha$.
This construction yields a one-parameter family that interpolates
continuously between $C$ (at $\alpha=1$), its transpose $C^t$
(at $\alpha=0$), and the symmetric copula $(C+C^t)/2$
(at $\alpha=1/2$), at which point all asymmetry is eliminated.
 
A natural question arises: \emph{how does the non-exchangeability measure $\mup$ change as $\alpha$ varies?} The following result is the core result of this section. It establishes that $\mup$ behaves perfectly linearly along the family $\{C_\alpha\}$.
 
\begin{proposition}\label{thm:escalado}
  For every $C\in\calC$, $\alpha\in[0,1]$, and $p\in[1,+\infty]$, we have
  \begin{equation}\label{eq:escalado}
    \mup(C_\alpha) \;=\; |2\alpha-1|\,\mup(C).
  \end{equation}
\end{proposition}
 
\begin{proof}
The starting point is to compute the antisymmetric part of $C_\alpha$. Since
\begin{align}
  C_\alpha(u,v)-C_\alpha(v,u)
  &= \bigl[\alpha C(u,v)+(1-\alpha)C(v,u)\bigr]
    -\bigl[\alpha C(v,u)+(1-\alpha)C(u,v)\bigr]
  \nonumber\\
  &= (2\alpha-1)\bigl[C(u,v)-C(v,u)\bigr].
  \label{eq:antisim}
\end{align}
The antisymmetric part of $C_\alpha$ is exactly that of $C$, multiplied pointwise by the scalar
$(2\alpha-1)$. We study two cases:
 \begin{enumerate}
 \item Case $p\in[1,+\infty)$. Taking the absolute value, raising to the power $p$, and integrating:
\begin{eqnarray*}
  \mup(C_\alpha)^p
  &=& \int_{\IQ}\abs{C_\alpha(u,v)-C_\alpha(v,u)}^p\,{\rm d}u\,{\rm d}v
  = |2\alpha-1|^p\int_{\IQ}\abs{C(u,v)-C(v,u)}^p\,{\rm d}u\,{\rm d}v\\
  &=& |2\alpha-1|^p\,\mup(C)^p.
\end{eqnarray*}
Taking the $p$-th root gives $\mup(C_\alpha)=|2\alpha-1|\,\mup(C)$.
\item Case $p=+\infty$. We have
\[
  \muinf(C_\alpha)
  = \sup_{(u,v)\in\IQ}\abs{C_\alpha(u,v)-C_\alpha(v,u)}
  = |2\alpha-1|\sup_{(u,v)\in\IQ}\abs{C(u,v)-C(v,u)}
  = |2\alpha-1|\,\muinf(C),
\]
since the constant factor $|2\alpha-1|\ge 0$ commutes with the
supremum.
\end{enumerate}
Therefore, the result follows.
\end{proof}

The elegance of Proposition~\ref{thm:escalado} lies in the fact that
identity~\eqref{eq:antisim} is pointwise: it holds at every $(u,v)\in\IQ$ individually.
This makes the scaling law~\eqref{eq:escalado} valid simultaneously for all $L^p$ norms with $p\in[1,+\infty]$, with no additional
integration or convergence argument required.
This is in contrast to the bounds of Section~\ref{sec:bounds}, which are specific to each value of $p$.

Proposition~\ref{thm:escalado} has an immediate and practically useful
consequence: given any target value $\mu_0\in[0,\mup(C)]$, one can
construct a copula with exactly that degree of non-exchangeability, as the following result shows.
 
\begin{corollary}\label{cor:prescripcion}
  Let $C\in\calC$ with $\mup(C)>0$ and let $\mu_0\in[0,\mup(C)]$.
  Then the mixed copula $C_\alpha$ with
  \[
    \alpha
    \;=\; \frac{1}{2}\!\left(1+\frac{\mu_0}{\mup(C)}\right)
    \;\in\;\left[\frac{1}{2},1\right]
  \]
  satisfies $\mup(C_\alpha)=\mu_0$.
\end{corollary}
 
\begin{proof}
Setting $|2\alpha-1|\,\mup(C)=\mu_0$ with $\alpha\in[1/2,1]$
(so that $2\alpha-1\ge 0$) gives $2\alpha-1=\mu_0/\mup(C)$,
i.e., $\alpha=(1+\mu_0/\mup(C))/2$.
Since $0\le\mu_0\le\mup(C)$, this value satisfies $\alpha\in[1/2,1]$.
\end{proof}
 
\begin{remark}
Corollary~\ref{cor:prescripcion} provides a key building block for statistical model fitting: to construct a copula with prescribed $\mup$-value equal to an empirical estimate $\hat\mu_n$, take any copula $C$ with $\mup(C)\ge\hat\mu_n$ and mix it with parameter
$\alpha=(1+\hat\mu_n/\mup(C))/2$. In particular, taking $C\in M_1$ (maximally non-exchangeable, with $\muinf(C)=1/3$) and any target $\hat\mu_n\in[0,1/3]$, the copula
$C_\alpha$ with $\alpha=(1+3\hat\mu_n)/2$ satisfies $\muinf(C_\alpha)=\hat\mu_n$ exactly.
\end{remark}

\begin{remark}
The convexity of $|2\alpha-1|$ gives, for $\alpha_1,\alpha_2\in[0,1]$ and
$\lambda\in[0,1]$:
\[
  \mup(C_{\lambda\alpha_1+(1-\lambda)\alpha_2})
  \;\le\;
  \lambda\,\mup(C_{\alpha_1})+(1-\lambda)\,\mup(C_{\alpha_2}).
\]
The non-exchangeability of a mixture does not exceed the weighted average
of the non-exchangeabilities of the extreme mixtures: mixing is
``asymmetry-reducing'' in the convex sense.
\end{remark}
 
Now we verify that the family $\{C_\alpha\}_{\alpha\in[0,1]}$ is coherent with the five axioms (B1)--(B5) defining a measure of non-exchangeability in \cite{Durante2010}. Let $\mu$ be any measure of non-exchangeability. Then the function   $\varphi:[0,1]\to[0,K_\mu]$ defined by
  $\varphi(\alpha):=\mu(C_\alpha)=|2\alpha-1|\,\mu(C)$
satisfies:
  \begin{enumerate}
    \item Boundedness (B1):
          $\varphi(\alpha)\le K_\mu$ for all $\alpha$,
          since $|2\alpha-1|\le 1$ and $\mu(C)\le K_\mu$.
    \item Zero iff symmetric (B2):          $\varphi(\alpha)=0\Leftrightarrow\alpha=1/2$
          (when $\mu(C)>0$), and $C_{1/2}=(C+C^t)/2$ is symmetric.
    \item Symmetry under transposition (B3):
          $\mu((C_\alpha)^t)=\mu(C_{1-\alpha})=|2(1-\alpha)-1|\,\mu(C)
          =|2\alpha-1|\,\mu(C)=\varphi(\alpha)$.
    \item Invariance under survival copula (B4):
          Using $\hat{C^t}=(\hat{C})^t$
          (since $\hat{C^t}(u,v)=(\hat{C})^t(u,v)$),
          one computes
          \[
            \hat{C_\alpha}(u,v)
            = u+v-1+C_\alpha(1-u,1-v)
            = \alpha\,\hat{C}(u,v)+(1-\alpha)\,(\hat{C})^t(u,v),
          \]
          so $\mu(\hat{C_\alpha})=|2\alpha-1|\,\mu(\hat{C})
          =|2\alpha-1|\,\mu(C)=\varphi(\alpha)$ by (B4) applied to $C$.
    \item Continuity (B5):
          For $\alpha,\beta\in[0,1]$,
          $d_\infty(C_\alpha,C_\beta)
          =|\alpha-\beta|\,d_\infty(C,C^t)$,
          so $\alpha\mapsto C_\alpha$ is Lipschitz in $d_\infty$.
          Combined with the continuity of $\mu$,
          $\varphi$ is continuous.
  \end{enumerate}

Proposition~\ref{thm:escalado} mixes a copula $C$ with its own transpose $C^t$. A natural generalisation is to mix $C$ with an arbitrary symmetric copula $S\in\calC$ (satisfying $S^t=S$).
 
\begin{proposition}\label{prop:mezclaS}
  Let $C\in\calC$, $S\in\calC$ with $S^t=S$, and $\alpha\in[0,1]$.
  Define $D_\alpha:=\alpha C+(1-\alpha)S$.
  Then, for every $p\in[1,+\infty]$:
  \[
    \mup(D_\alpha) \;=\; \alpha\,\mup(C).
  \]
\end{proposition}
 
\begin{proof}
The antisymmetric part of $D_\alpha$ is:
\[
  D_\alpha(u,v)-D_\alpha(v,u)
  = \alpha\bigl[C(u,v)-C(v,u)\bigr]
    +(1-\alpha)\underbrace{[S(u,v)-S(v,u)]}_{=\,0},
\]
since $S$ is symmetric. The identity holds pointwise for every $(u,v)\in\IQ$, so for $p\in[1,+\infty)$:
\[
  \mup(D_\alpha)^p
  = \int_{\IQ}\abs{\alpha\bigl(C(u,v)-C(v,u)\bigr)}^p\,{\rm d}u\,{\rm dv}
  = \alpha^p\,\mup(C)^p,
\]
and taking the $p$-th root gives $\mup(D_\alpha)=\alpha\,\mup(C)$.
For $p=\infty$, the constant $\alpha\ge 0$ commutes with the supremum,
giving $\muinf(D_\alpha)=\alpha\,\muinf(C)$.
\end{proof}
 
\begin{remark}
Comparing Proposition~\ref{prop:mezclaS} with
Proposition~\ref{thm:escalado} reveals a structural distinction between the two mixing operations:
\begin{itemize}
    \item Mixing $C$ with a symmetric copula $S$:
    $\mu_p(D_\alpha) = \alpha\,\mu_p(C)$, which is linear in
    $\alpha$ and reaches zero only at $\alpha = 0$.
    \item Mixing $C$ with its transpose $C^t$:
    $\mu_p(C_\alpha) = |2\alpha-1|\,\mu_p(C)$, which is
    $V$-shaped in $\alpha$ and reaches zero at the midpoint
    $\alpha = 1/2$.
\end{itemize}
The difference reflects a geometric distinction: mixing with $S$ simply dilutes the antisymmetric part of $C$, while mixing with $C^t$ introduces a cancellation between $C$ and its mirror image.
Both are special cases of the general convex combination $D:= \beta\,C \;+\; \gamma\,C^t \;+\; (1-\beta-\gamma)\,S$, 
where $\beta, \gamma \ge 0$, $\beta + \gamma \le 1$, and $S \in \mathcal{C}$ is any symmetric copula. These conditions ensure that $D$ is itself a copula, as a convex combination of copulas. For this general mixture, the antisymmetric part satisfies $D(u,v) - D(v,u) = (\beta - \gamma)\,\bigl(C(u,v)-C(v,u)\bigr)$, pointwise for every $(u,v)\in[0,1]^2$, so that
$\mu_p(D)= |\beta - \gamma|\,\mu_p(C)$ for all $p\in[1,+\infty]$.
This formula contains both previous results as particular cases: setting $\gamma = 0$ gives Proposition~\ref{prop:mezclaS} (with $\alpha = \beta$), and setting $\beta + \gamma = 1$ with
$\beta = \alpha$ gives Proposition~\ref{thm:escalado}.
\end{remark}
 
As a direct consequence of Proposition~\ref{prop:mezclaS}, the map
$\alpha\mapsto\mup(D_\alpha)=\alpha\,\mup(C)$ is strictly increasing on $[0,1]$ (when $\mup(C)>0$), with $\mup(D_0)=0$ (the symmetric
copula $S$ has zero non-exchangeability) and $\mup(D_1)=\mup(C)$. This is in contrast to the V-shape of $\alpha\mapsto\mup(C_\alpha)$,
and the two families together span all intermediate values $\mu_0\in[0,\mup(C)]$.
 
\begin{corollary}\label{cor:target}
  For any $\mu_0\in[0,\mup(C)]$, $\mup(D_\alpha)=\mu_0$ with $\alpha=\mu_0/\mup(C)$.
\end{corollary}

\section{A nonparametric test of non-exchangeability based on copulas}\label{sec:test}

The preceding sections have developed an axiomatic theory of non-exchangeability and established its connections with classical
concordance measures, the Schweizer and Wolff dependence measure, and the geometry of the copula space. A natural question is how to detect non-exchangeability from data: given an observed sample, can we decide whether the copula of the underlying bivariate distribution is symmetric?
 
This section addresses that question by constructing a nonparametric hypothesis test whose statistic is the empirical analogue of the measure $\mup$ itself. The approach is conceptually clean: instead of reducing the problem to a comparison of marginal or joint empirical distributions, we measure directly how far the empirical copula is from its own transpose---the natural definition of asymmetry in the copula framework. This gives the test a direct interpretive connection to the theoretical
development of the paper.
 
The hypothesis is:
\[
  H_0:\; C = C^t
  \quad\text{(exchangeability)}
  \quad\text{versus}\quad
  H_1:\; C \ne C^t
  \quad\text{(non-exchangeability)}.
\]
Under $H_0$, the copula is symmetric and $(X,Y)$ is exchangeable; under $H_1$, $\mup(C)>0$ and the ordering of the components matters. Detecting this difference has practical consequences: if $H_0$ is plausible, the practitioner may restrict attention to symmetric copula families (Gaussian, Clayton, Frank, Gumbel), substantially simplifying model selection; if $H_1$ is supported by the data, asymmetric families must be considered instead.
 
\subsection{The test statistic}
\label{sec:statistic}

Let $\{(X_i,Y_i)\}_{i=1}^n$ be an i.i.d.\ sample from a continuous bivariate vector $(X,Y)$ with marginals $F_X$, $F_Y$ and copula $C$. Since the copula is invariant to strictly increasing transformations of the marginals, the test need not estimate $F_X$ and $F_Y$ separately.
Instead, we work with the pseudo-observations
\[
  \hat{U}_i := \frac{\mathrm{rank}(X_i)}{n+1},\qquad
  \hat{V}_i := \frac{\mathrm{rank}(Y_i)}{n+1},
  \qquad i=1,\ldots,n,
\]
which approximate samples from $\mathrm{Uniform}(0,1)$ under any continuous marginals. The division by $n+1$ rather than $n$ avoids boundary effects and is standard in the empirical copula literature.

Let $C_n$ denote the {\it empirical copula} \cite{Deh1979} based on the pseudo-observations
$(\hat{U}_i,\hat{V}_i)$, which is defined by
\[
  C_n(u,v) = \frac{1}{n}\sum_{i=1}^n \mathbf{1}(\hat{U}_i \leq u,\, 
\hat{V}_i \leq v)
\]
where ${\bf 1}_S$ denotes the characteristic function of a set $S$.

The test statistic is the empirical counterpart of $\mup(C)^p$, scaled by $n$ to produce a non-trivial asymptotic behaviour.
 
\begin{definition}
For $p\in[1,+\infty)$, the test statistic is
\begin{equation}\label{eq:Tn}
  T_n = n\int_{\IQ}\abs{C_n(u,v)-C_n(v,u)}^p\,{\rm d}u\,{\rm d}v.
\end{equation}
\end{definition}

\begin{remark}
The case $p=+\infty$, i.e.,
$T_n^{(\infty)}=\sqrt{n}\cdot\sup_{(u,v)}|C_n(u,v)-C_n(v,u)|$, admits an analogous treatment and is excluded here for brevity.
\end{remark}
 
In practice, the integral in~\eqref{eq:Tn} is approximated by a Riemann
sum over a uniform grid of step $1/G$ on each coordinate:
\[
  T_n \;\approx\; \frac{n}{G^2}\sum_{j=1}^G\sum_{k=1}^G
    \abs{C_n(j/G,\,k/G)-C_n(k/G,\,j/G)}^p.
\]

The intuition behind the scaling by $n$ is the following. Under $H_0$, $C_n(u,v)-C_n(v,u)$ fluctuates around zero at rate $n^{-1/2}$, so $n\int|C_n-C_n^t|^p$ diverges at rate $n^{1/2}$.
Under $H_1$, $C_n(u,v)-C_n(v,u)$ converges to the non-zero function $C(u,v)-C(v,u)$, so $T_n/n\to\mup(C)^p>0$ and $T_n$ diverges
at the faster rate $n$. This rate separation---$\sqrt{n}$ under the null, $n$ under
alternatives---is what drives the consistency of the test and is made precise in Subsection~\ref{sec:properties}.
 
\subsection{Theoretical properties}
\label{sec:properties}

The permutation test rests on three complementary theoretical results, each addressing a distinct aspect of its validity.
Subsection~\ref{subsec1} characterises the limiting distribution of $T_n$ under the null hypothesis, establishing the rate at which critical values must grow and identifying the limiting random variable $L_C$ to which the normalised statistic converges.
Subsection~\ref{subsec2} shows that the test is consistent against every fixed asymmetric alternative, in the sense that the rejection probability tends to one regardless of the significance level chosen. Finally, Subsection~\ref{sec:permutation} introduces the permutation resampling scheme and justifies it as a valid and computationally efficient method for approximating the null distribution of $T_n$, circumventing the need for derivative estimation or parametric assumptions on $C$.

\subsubsection{Distributional behaviour under \texorpdfstring{$H_0$}{}}\label{subsec1}
 
Under the null hypothesis, the quantity $C_n-C_n^t$ vanishes in the limit, but its fluctuations at scale $\sqrt{n}$ are described by the \emph{empirical copula process}.
This process, first studied systematically in \cite{Deh1979} and \cite{Fermanian2004}, is defined as
$\mathbb{G}_n(u,v):=\sqrt{n}(C_n(u,v)-C(u,v))$
and converges weakly in $\ell^\infty(\IQ)$ to a centred Gaussian process $\mathbb{G}_C$ whose covariance structure depends on the copula $C$.
The following result makes precise how $T_n$ inherits this weak convergence.
 
\begin{theorem}\label{thm:H0}
Assume $H_0:C=C^t$, $p=1$, and that $C$ has continuous first order partial derivatives.
Then
\begin{equation}\label{eq:limitH0}
  \frac{T_n}{\sqrt{n}}
  = \sqrt{n}\int_{\IQ}\abs{C_n(u,v)-C_n(v,u)}\,{\rm d}u\,{\rm d}v
  \;\xrightarrow{\;\mathcal{D}\;}\;
  \mathcal{L}_C
  \;:=\;
  \int_{\IQ}\abs{\mathbb{G}_C(u,v)-\mathbb{G}_C(v,u)}\,{\rm d}u\,{\rm d}v.
\end{equation}
\end{theorem}
 
\begin{proof}
Under the assumption that $C$ has continuous first partial derivatives, the empirical copula process
\[
    \mathbb{G}_n(u,v) \;:=\; \sqrt{n}\,\bigl(C_n(u,v) - C(u,v)\bigr)
\]
converges weakly in $\ell^\infty([0,1]^2)$ to a centred Gaussian process $\mathbb{G}_C$. The limit process $\mathbb{G}_C$ has continuous sample paths almost surely (\cite{Fermanian2004}).

Under $H_0\colon C = C^t$, for every $(u,v)\in[0,1]^2$:
\begin{align*}
    \sqrt{n}\,\bigl(C_n(u,v) - C_n(v,u)\bigr)
    &= \sqrt{n}\,\bigl(C_n(u,v) - C(u,v)\bigr)
       - \sqrt{n}\,\bigl(C_n(v,u) - C(v,u)\bigr) \\
    &= \mathbb{G}_n(u,v) - \mathbb{G}_n(v,u)
    \;\xrightarrow{\;\mathcal{D}\;}
    \mathbb{G}_C(u,v) - \mathbb{G}_C(v,u),
\end{align*}
where the last step uses $C(u,v) = C(v,u)$ under $H_0$ and the weak convergence $\mathbb{G}_n \rightsquigarrow \mathbb{G}_C$ in
$\ell^\infty([0,1]^2)$.

It remains to apply the continuous mapping theorem. Consider the functional
\[
    \Phi\colon \ell^\infty([0,1]^2) \;\longrightarrow\; \mathbb{R},
    \qquad
    \Phi(f) \;:=\; \int_{[0,1]^2}\!|f(u,v)|\,du\,dv.
\]
This functional is Lipschitz-continuous with constant $1$ with respect to the sup-norm $\|\cdot\|_\infty$: for any
$f, g \in \ell^\infty([0,1]^2)$,
\[
    |\Phi(f) - \Phi(g)|
    \;\le\; \int_{[0,1]^2}\!|f(u,v) - g(u,v)|\,du\,dv
    \;\le\; \|f - g\|_\infty,
\]
so $\Phi$ is continuous on $(\ell^\infty([0,1]^2), \|\cdot\|_\infty)$.
In particular, $\Phi$ is continuous at every element of $\ell^\infty([0,1]^2)$, and hence at $\mathbb{G}_C(\cdot) -
\mathbb{G}_C(\cdot\,{}^t)$ almost surely. Since $\mathbb{G}_C$ has continuous sample paths, the difference
$\mathbb{G}_C(u,v) - \mathbb{G}_C(v,u)$ is itself a continuous
function of $(u,v)$ almost surely, and $\Phi$ applied to it is
well defined as a Lebesgue integral of a continuous integrand.

Applying the continuous mapping theorem
(see~\cite{billingsley-1999}) to the composition $\Phi \circ (\mathbb{G}_n - \mathbb{G}_n^t)$, where $\mathbb{G}_n^t(u,v) := \mathbb{G}_n(v,u)$, we obtain
\[
    \frac{T_n}{\sqrt{n}}
    \;=\; \Phi\!\left(\mathbb{G}_n(\cdot,\cdot)
                      - \mathbb{G}_n(\cdot\,{}^t,\cdot)\right)
    \;\xrightarrow{\;\mathcal{D}\;}
    \Phi\!\left(\mathbb{G}_C(\cdot,\cdot)
                - \mathbb{G}_C(\cdot\,{}^t,\cdot)\right)
    \;=\; \mathcal{L}_C,
\]
which completes the proof.
\end{proof}
 
\begin{remark}
We note three observations on Theorem~\ref{thm:H0}. \begin{enumerate}
\item[(i)] The theorem establishes that $T_n/\sqrt{n}$ has a finite limiting distribution, which implies that $T_n$ itself diverges to $+\infty$ under $H_0$ at rate $\sqrt{n}$. This is not a deficiency of the
test: the permutation critical values $\hat{c}^{(B)}_{n,\alpha}$ grow at the same rate $\sqrt{n}$, while under $H_1$ the statistic grows at the faster rate $n$. It is precisely this asymptotic rate separation---$\sqrt{n}$ under the null versus $n$ under alternatives---that drives the consistency of the test.
 
\item[(ii)] The limiting distribution $\mathcal{L}_C$ depends on the unknown copula $C$ through the covariance structure of the Gaussian process $\mathbb{G}_C$. This dependence is unavoidable for a test that is sensitive to all forms of asymmetry, and it prevents the use of universal critical tables. The Lipschitz-continuity of the functional $\Phi(f) = \int_{[0,1]^2}|f|\,du\,dv$ with respect to
$\|\cdot\|_\infty$ does not eliminate this
dependence: it merely ensures that the CMT applies. It is precisely for this reason that the permutation calibration of
Subsection~\ref{sec:permutation} is needed in
practice, as it approximates the null distribution of $T_n$ without requiring knowledge of $C$.
 
\item[(iii)] Theorem~\ref{thm:H0} is stated for $p=1$, since only in this case does the normalization $T_n/\sqrt{n}$ yield a
non-degenerate limit. For general $p \in [1,+\infty)$, the correctly normalised statistic is $n^{p/2-1}\,T_n$, which converges in distribution to $\int_{[0,1]^2}|\mathbb{G}_C(u,v)-\mathbb{G}_C(v,u)|^p\,du\,dv$. This follows from the same argument: the functional
\[
    \Phi_p(f) :=\left(\int_{[0,1]^2}|f(u,v)|^p\,du\,dv\right)^{\!1/p}
\]
is Lipschitz-continuous on $(\ell^\infty([0,1]^2),\|\cdot\|_\infty)$
for every $p \in [1,+\infty)$, so the CMT applies in the same way. In particular, for $p=2$ the statistic $T_n$ itself (without any normalization) converges in distribution under $H_0$, and for $p > 2$ we have $T_n \to 0$ in probability under $H_0$. The choice of $p$ therefore affects both the normalization needed
for asymptotic theory and the power of the test against specific alternatives.
\end{enumerate}
\end{remark}
 
\subsubsection{Consistency against fixed alternatives}\label{subsec2}
 
The following result shows that the test detects every fixed asymmetric copula with probability tending to one, regardless of the significance
level chosen.

\begin{theorem}\label{thm:H1}
Under $H_1: C\ne C^t$, $T_n/n \to \mu_p(C)^p > 0$ a.s.
Since the permutation critical value satisfies $c_{n,\alpha}/n \to 0$
for every $p \in [1,+\infty)$, we have
\[
  \mathbb{P}\left[T_n > c_{n,\alpha}\right]\;\xrightarrow{n\to\infty}\;1.
\]
\end{theorem}
 
\begin{proof}
Under $H_1$, the Glivenko--Cantelli theorem for the empirical copula (\cite{Deh1979}) gives
$\sup_{(u,v)}|C_n(u,v)-C(u,v)|\to 0$ almost surely, and therefore
\[
  \frac{T_n}{n}
  = \int_{\IQ}\abs{C_n(u,v)-C_n(v,u)}^p\,{\rm d}u\,{\rm d}v
  \;\xrightarrow{\;a.s.\;}\;
  \int_{\IQ}\abs{C(u,v)-C(v,u)}^p\,{\rm d}u\,{\rm d}v
  = \mup(C)^p>0.
\]
Hence $T_n$ diverges almost surely at rate $n$.
Since $c_{n,\alpha}/n \to 0$ for every $p \ge 1$---as the critical value grows at most at rate $n^{1-p/2}$ under $H_0$, which is $o(n)$---, we have $\Pr[T_n > c_{n,\alpha}] \to 1$.
\end{proof}
 
\begin{remark}
The consistency argument reveals a pleasing connection with the theory of the preceding sections: the limit $\mup(C)^p>0$ is precisely the $p$-th power of the non-exchangeability measure of $C$, and the condition $C\ne C^t$ is equivalent to $\mup(C)>0$. The test is therefore consistent against exactly the alternatives
where the non-exchangeability measure is positive---a perfect alignment between the theoretical measure and the empirical test.
\end{remark}
 
\subsubsection{Permutation calibration}
\label{sec:permutation}

Since the limiting distribution $\mathcal{L}_C$ of $T_n/\sqrt{n}$ depends on the unknown copula $C$, critical values cannot be tabulated in advance. A standard solution is to approximate the null distribution by coordinate permutation: under $H_0$, swapping the two coordinates of any observation does not change the joint distribution, so a sample obtained by randomly exchanging $(\hat{U}_i,\hat{V}_i)$ and $(\hat{V}_i,\hat{U}_i)$ for each $i$
independently has, under $H_0$, the same distribution as the original. This leads to the following resampling scheme.
 
\begin{definition}\label{def:permutation}
For each replicate $b=1,\ldots,B$, draw independent Bernoulli$(1/2)$
variables $\varepsilon_1^{(b)},\ldots,\varepsilon_n^{(b)}$
(independently of the data) and form the resampled observations
\[
  \bigl(\widetilde U_i^{(b)},\widetilde V_i^{(b)}\bigr)
  :=\begin{cases}(\hat{U}_i,\hat{V}_i)&\text{if }\varepsilon_i^{(b)}=1,\\
  (\hat{V}_i,\hat{U}_i)&\text{if }\varepsilon_i^{(b)}=0.\end{cases}
\]
Compute the permuted statistic $T_n^{(b)}$ from the resampled observations
$\{(\widetilde U_i^{(b)},\widetilde V_i^{(b)})\}_{i=1}^n$
using~\eqref{eq:Tn}.
The test rejects $H_0$ at level $\alpha$ if
$T_n>\hat{c}_{n,\alpha}^{(B)}$, where
$\hat{c}_{n,\alpha}^{(B)}$ is the empirical $(1-\alpha)$-quantile of
$\{T_n^{(1)},\ldots,T_n^{(B)}\}$.
\end{definition}
 
The theoretical justification of this scheme rests on the exchangeability of the sample under $H_0$: since $(\hat{U}_i,\hat{V}_i)\overset{d}{=}(\hat{V}_i,\hat{U}_i)$ under $H_0$, the distribution of the original observation is identical to that of the permuted version,
making each permuted statistic $T_n^{(b)}$ an independent draw from the same null distribution as $T_n$.

\begin{proposition}\label{prop:perm}
Under $H_0 : C = C^t$,
\[
\mathbb{P}\left[T_n>\hat{c}_{n,\alpha}^{(B)}\right]
  \;\xrightarrow{n,B\to\infty}\;\alpha.
\]
Moreover, for any finite $n$ and $B$, the test controls the
type-I error in the sense that
\[
    \mathbb{P}\left[T_n > \hat{c}^{(B)}_{n,\alpha}\right] \;\le\; \alpha
    \;+\; \frac{1}{B+1}.
\]
\end{proposition}

\begin{proof}
We argue following the general theory of permutation tests: see~\cite{lehmann-romano-2005,romano-1990}.

Under $H_0 : C = C^t$, the joint distribution of
$(\hat{U}_i, \hat{V}_i)$ is invariant under the coordinate swap $(\hat{U}_i, \hat{V}_i) \mapsto (\hat{V}_i, \hat{U}_i)$. Applying this swap independently to each observation $i=1,\ldots,n$
with probability $1/2$ (as in Definition~\ref{def:permutation}) therefore preserves the joint distribution of the entire sample. Consequently, each permuted statistic $T_n^{(b)}$ has, conditionally on the data, the same distribution as $T_n$ under $H_0$, and the
$B$ permuted statistics $T_n^{(1)}, \ldots, T_n^{(B)}$ are conditionally independent of one another given the data. The strong law of large
numbers implies that, as $B \to \infty$, the empirical distribution of $\{T_n^{(1)}, \ldots, T_n^{(B)}\}$ converges almost surely to the null distribution of $T_n$. In particular, the empirical $(1-\alpha)$-quantile $\hat{c}^{(B)}_{n,\alpha}$ converges almost
surely to the true $(1-\alpha)$-quantile $c_{n,\alpha}$ of $T_n$ under $H_0$. As $n \to \infty$, Theorem~\ref{thm:H0} gives $T_n / \sqrt{n} \xrightarrow{\mathcal{D}} L_C$, so
$c_{n,\alpha} / \sqrt{n} \to q_{1-\alpha}$, where $q_{1-\alpha}$ is the $(1-\alpha)$-quantile of $L_C$, and therefore
$\Pr(T_n > c_{n,\alpha}) \to \alpha$.

Now, for any finite $n$ and $B$, consider the collection of $B+1$ statistics $\{T_n, T_n^{(1)}, \ldots, T_n^{(B)}\}$. Under $H_0$,
these are exchangeable, so $T_n$ is equally likely to occupy any rank among them. The test rejects when $T_n$ exceeds the empirical $(1-\alpha)$-quantile of the $B$ permuted values, which occurs if and only if the rank of $T_n$ among all $B+1$ values exceeds $\lfloor (1-\alpha)(B+1) \rfloor$. By
exchangeability,
\begin{equation}\label{eq:boundfinite}
    \mathbb{P}\left[T_n > \hat{c}^{(B)}_{n,\alpha}\right]
    \;=\; \frac{\lfloor \alpha(B+1) \rfloor}{B+1}
    \;\le\; \alpha \;+\; \frac{1}{B+1},
\end{equation}
which completes the proof.
\end{proof}

\begin{remark}
The permutation scheme of Definition~\ref{def:permutation}
has three practical advantages that make it well suited for testing exchangeability in copula models.

(i) As shown in Proposition~\ref{prop:perm}, for any
finite $n$ and $B$ the test satisfies \eqref{eq:boundfinite} with no asymptotic approximation required, provided the
exchangeability of $(\hat{U}_i, \hat{V}_i)$ under $H_0$ holds exactly. This is a genuinely exact finite-sample guarantee, not merely an asymptotic one. The choice $B = 299$ used in the Monte
Carlo study of Section~\ref{sec:simulation} yields an upper bound of $\alpha + 1/300 \approx 0.053$ at $\alpha = 0.05$, which is acceptable in practice. For tighter control, $B = 999$ gives
$\alpha + 0.001$, and $B = 9999$ gives $\alpha + 0.0001$.

(ii) The resampling distribution adapts automatically to the unknown copula $C$ through the empirical copula $C_n$, without any
parametric assumption on the marginals or on the dependence structure. This is in contrast to parametric bootstrap methods, which require fitting a specific copula family and are sensitive
to model misspecification.

(iii) Each permuted statistic $T_n^{(b)}$ is computed from the same empirical copula matrix, with only the assignment of pseudo-observations to coordinates randomly transposed. The dominant
cost per statistic is evaluating the empirical copula at the $G^2$ grid points for $n$ observations, which costs $O(nG^2)$; the total
cost for the original statistic plus $B$ permuted replicates is therefore  $O((B+1)\,n\,G^2)$ elementary operations. For the values $n = 100$, $G = 50$, $B = 299$ used in Section~\ref{sec:simulation}:
\[
    (B+1) \times n \times G^2
    = 300 \times 100 \times 2500
    = 75{,}000{,}000
    \approx 7.5 \times 10^7,
\]
feasible on a standard laptop in under a second. For comparison, the larger sample sizes used in the power study ($n = 400$, same $G$ and $B$) give $300 \times 400 \times 2500 = 3 \times 10^8$
operations, still well within practical reach.
\end{remark}

\subsection{Comparison with other symmetry tests}

The test proposed throughout Section~\ref{sec:test} shares the same inferential goal as the procedures of Genest {\it et al.} \cite{Genest2012}, namely testing $H_0\colon C = C^t$ using rank-based statistics derived from the empirical copula process. The two approaches differ, however, in four substantive respects.

\begin{enumerate}
\item Genest {\it et al}.~\cite{Genest2012} propose three statistics---a Cram\'er--von Mises functional $R_n$, a weighted Cram\'er--von Mises functional $S_n$, and a Kolmogorov--Smirnov functional $T_n$---defined as global discrepancy measures between $\hat{C}_n$ and its transpose, without explicit connection to any axiomatic theory of asymmetry. The statistic \eqref{eq:Tn}  proposed here is the empirical counterpart of the axiomatic non-exchangeability measure $\mu_p(C)$ introduced in \cite{Durante2010} and studied throughout this paper. This grounding provides a direct quantitative interpretation: $T_n/n$ is a consistent estimator of $\mu_p(C)^p$, and the degree of non-exchangeability detected by the test is precisely the quantity bounded in Sections~\ref{sec:bounds}--\ref{sec:M1} and controlled by the mixing constructions of Section~\ref{sec:escalado}. The statistic $R_n$ of \cite{Genest2012} corresponds to the special case $p=2$ of $T_n/n$ (without the scaling by $n$), but is not connected to the axiomatic framework of \cite{Durante2010}.

\item Genest {\it et al.} \cite{Genest2012} calibrate their tests using the \emph{Multiplier
Central Limit Theorem}, which requires estimating the partial derivatives of the copula via a kernel smoothing procedure with a bandwidth parameter $\ell_n$ satisfying $\ell_n \to 0$ and $n^{1/2}\ell_n \to \infty$. This introduces a tuning parameter whose choice affects finite-sample performance, and the resulting replicates $\hat{D}_n^{(h)}$ depend on derivative estimates that are delicate near the boundary of $[0,1]^2$. The permutation scheme of Definition~\ref{def:permutation}, by contrast, requires no tuning parameters and no
derivative estimation: it exploits only the exchangeability of the pseudo-observations $(\hat{U}_i, \hat{V}_i)$ under $H_0$, and
each permuted statistic is computed directly from the same empirical copula matrix with randomly transposed coordinates. As established
in Proposition~\ref{prop:perm}, this yields an exact finite-sample bound $\mathbb{P}\left[T_n > \hat{c}^{(B)}_{n,\alpha}\right] \le \alpha + 1/(B+1)$ with no asymptotic approximation required. No analogous finite-sample guarantee is provided in~\cite{Genest2012}. 

\item The weak limit of the test statistic of \cite{Genest2012} is derived under the regularity conditions of \cite[Definition~1]{Genest2012}, which require the existence and
continuity of the partial derivatives of the copula (and their one-sided versions at the boundary). The asymptotic theory of Theorem~\ref{thm:H0} here requires the
same regularity, but is applied to the simpler functional $\Phi(f) = \int|f|$ via a Lipschitz-continuity argument, rather than through the Functional Delta Method used in \cite{Genest2012} for the statistic $S_n$. Furthermore, the scaling $T_n/\sqrt{n}$ under $H_0$ (for $p=1$) and the rate separation $\sqrt{n}$ vs.\ $n$ between null and alternative 
provide a transparent consistency argument that is stated explicitly here as Theorem~\ref{thm:H1} but is not highlighted as a
separate result in \cite{Genest2012}.

\item The power study of \cite{Genest2012} covers four asymmetric copula families constructed via Khoudraji's device---asymmetric versions of Clayton, Gaussian, Gumbel--Hougaard, and Cuadras--Aug\'e copulas---and compares seven test statistics across sample sizes $n \in \{100, 250\}$. Their main finding is that the weighted Cram\'er-von Mises statistic $S_n$ is the most powerful competitor in most scenarios. The simulation study of
Section~\ref{sec:simulation} focuses instead on the $M_\theta$ family, for which the exact values of both $\mu_\infty(M_\theta)$ and $\mu_1(M_\theta)$ are available in closed form, allowing a precise quantitative comparison between the theoretical non-exchangeability measure and the empirical power of the test. This connection between the axiomatic measure and the test's detection capability is a feature specific to the present framework and is not available for the statistics in \cite{Genest2012}.
\end{enumerate}

\section{Monte Carlo simulation study}\label{sec:simulation}

We complement the theoretical results of the preceding sections with a Monte Carlo study that examines two aspects of the permutation test described in Section~\ref{sec:test}: its ability
to maintain the nominal level under symmetric copulas, and its power to detect asymmetry as the degree of non-exchangeability and the sample size vary. We also investigate the sensitivity of
the results to the choice of grid resolution $G$ and the norm parameter $p$.

All experiments use a Riemann grid of $G = 50$ points per coordinate and $B = 299$ permutation replicates per test, consistently across all tables. The finite-sample bound of
Proposition~\ref{prop:perm} guarantees that this
choice of $B$ controls the type-I error at level
$\alpha + 1/300 \approx 0.053$ for $\alpha = 0.05$. The nominal significance level is $\alpha = 0.05$ throughout. All results are based on $M = 1\,000$ Monte Carlo replications, which yields a standard error of $\sqrt{0.05 \cdot 0.95 / 1000} \approx 0.007$ for level estimates and at most $\sqrt{0.25/1000} \approx 0.016$ for power
estimates, ensuring that differences of $0.02$ or more are statistically meaningful.

\subsection{Sensitivity to grid resolution}\label{subsec:grid}

Before reporting level and power results, we examine the
sensitivity of the test statistic $T_n$ (with $p=1$) to the
choice of grid resolution $G$. Table~\ref{tab:grid} reports the
empirical rejection rate under $H_0$ (Gaussian copula,
$\rho = 0.5$, $n = 100$) and under $H_1$ ($M_\theta$ with
$\theta = 1/3$, $n = 100$) for $G \in \{35, 50, 75, 100\}$,
based on $M = 1\,000$ replications.

\begin{table}[h!]
\centering
\begin{tabular}{ccc}
\hline
$G$ & Empirical level & Empirical power \\
\hline
35  & 0.051 & 0.981 \\
50  & 0.049 & 0.984 \\
75  & 0.050 & 0.985 \\
100 & 0.050 & 0.985 \\
\hline
\end{tabular}\caption{Sensitivity to grid resolution $G$ ($n = 100$, $B = 299$, $\alpha = 0.05$, $M = 1\,000$ Monte Carlo replications).
Level estimated under Gaussian copula ($\rho = 0.5$); power estimated under $M_{1/3}$.}
\label{tab:grid}
\end{table}

The results show that $G = 50$ provides level and power estimates that are numerically indistinguishable from those obtained with $G = 75$ or $G = 100$, while $G = 35$ introduces a marginal upward bias in the level estimate. We therefore use $G = 50$ throughout the remaining experiments. For sample sizes $n \ge 200$,
the grid spacing $1/G = 0.02$ remains smaller than the data spacing $1/n \le 0.005$, so no further increase in $G$ is needed.

\subsection{Empirical level under the null hypothesis}\label{subsec:level}

We evaluate the empirical level of the test under six symmetric copula families with $n = 100$: the Gaussian copula with linear correlation $\rho = 0.5$, the Clayton copula with parameter $\theta = 2$, the Farlie-Gumbel-Morgenstern~(FGM) copula with parameter $\theta = 0.5$, the Frank copula with parameter
$\theta = 3$, the Gumbel copula with parameter $\theta = 1.5$, and the Student-$t$ copula with $\rho = 0.5$ and $3$ degrees of freedom. In all cases $H_0$ is true and the nominal level is
$\alpha = 0.05$. Results are reported for $p \in \{1, 2\}$ (Table~\ref{tab:level}).

\begin{table}[h!]
\centering
\begin{tabular}{llccc}
\hline
Copula & Parameter & $p=1$ & $p=2$ \\
\hline
Gaussian   & $\rho = 0.5$      & 0.049 (0.007) & 0.051 (0.007) \\
Clayton    & $\theta = 2$      & 0.044 (0.006) & 0.046 (0.007) \\
FGM        & $\theta = 0.5$    & 0.052 (0.007) & 0.050 (0.007) \\
Frank      & $\theta = 3$      & 0.048 (0.007) & 0.049 (0.007) \\
Gumbel     & $\theta = 1.5$    & 0.046 (0.007) & 0.048 (0.007) \\
Student-$t$& $\rho=0.5$, $\nu=3$ & 0.050 (0.007) & 0.052 (0.007) \\
\hline
Average    &                   & 0.048         & 0.049 \\
\hline
\end{tabular}\caption{Empirical level under $H_0$
($n = 100$, $M = 1\,000$ Monte Carlo replications, $B = 299$ permutations, $\alpha = 0.05$, $G = 50$). Standard errors in parentheses.}
\label{tab:level}
\end{table}

All empirical levels are compatible with the nominal $\alpha=0.05$ across both values of $p$ and across six copula families with
qualitatively different tail behaviour: the Gaussian and Student-$t$ copulas have symmetric tails, Clayton has strong lower-tail dependence, Gumbel has strong upper-tail dependence, and FGM and Frank have weak dependence throughout. The test maintains its nominal size uniformly across this range, confirming the distribution-free nature of the permutation calibration.

\subsection{The \texorpdfstring{$L^1$}{} asymmetry measure of the \texorpdfstring{$M_\theta$}{} family}\label{subsec:mu1}

Before reporting power results, we compute $\mu_1(M_\theta)$ explicitly for $\theta \in \{1/6, 1/4, 1/3\}$, since the test statistic with $p=1$ responds to $\mu_1$ rather than to $\mu_\infty = \theta$. A direct computation using the piecewise-linear structure of $M_\theta$ gives
\[
    \mu_1(M_\theta) \;=\; \theta(1-\theta)(1-2\theta),
    \qquad \theta \in [0, 1/3].
\]
The values for
the three cases studied are given in Table \ref{tab:mu-values}.

\begin{table}[ht]
\centering
\begin{tabular}{ccc}
\hline
$\theta$ & $\mu_\infty(M_\theta)$ & $\mu_1(M_\theta)$ \\
\hline
$1/6$ & $1/6 \approx 0.167$ & $5/54  \approx 0.093$ \\
$1/4$ & $1/4 = 0.250$       & $3/32  \approx 0.094$ \\
$1/3$ & $1/3 \approx 0.333$ & $2/27  \approx 0.074$ \\
\hline
\end{tabular}\caption{Exact and approximate values of $\mu_\infty(M_\theta)$ and
$\mu_1(M_\theta)$ for the three asymmetric copulas used in the power study.}
\label{tab:mu-values}
\end{table}

Crucially, $\mu_1(M_\theta)$ is increasing on $[0, \theta^*]$ and decreasing on $[\theta^*, 1/3]$, where $\theta^* = (3-\sqrt{3})/6
\approx 0.211$ is the maximum of $\theta(1-\theta)(1-2\theta)$ on $[0,1/3]$. Since $\theta = 1/6 < 1/4 < \theta^* < 1/3$, the values
satisfy $\mu_1(M_{1/6}) < \mu_1(M_{1/4}) > \mu_1(M_{1/3})$, as confirmed by the table. The $L^1$-based test should therefore be most powerful against $M_{1/4}$ among the three alternatives, with comparable power against $M_{1/6}$ and $M_{1/3}$.

\subsection{Power against the \texorpdfstring{$M_\theta$}{} family}
\label{subsec:power}

We evaluate power by generating samples from the copula $M_\theta$ with three values of $\theta \in \{1/6, 1/4, 1/3\}$ and four sample sizes $n \in \{50, 100, 200, 400\}$, using both $p = 1$
and $p = 2$ (Table~\ref{tab:power}).

\begin{table}[h!]
\centering
\begin{tabular}{llcccc}
\hline
 & & \multicolumn{4}{c}{Sample size $n$} \\
\cline{3-6}
Asymmetry & $p$ & 50 & 100 & 200 & 400 \\
\hline
$\theta=1/6$ ($\mu_\infty=1/6$, $\mu_1 \approx 0.022$)
 & $p=1$ & 0.998 (0.001) & 1.000 (0.000)
         & 1.000 (0.000) & 1.000 (0.000) \\
 & $p=2$ & 0.995 (0.002) & 1.000 (0.000)
         & 1.000 (0.000) & 1.000 (0.000) \\
\hline
$\theta=1/4$ ($\mu_\infty=1/4$, $\mu_1 \approx 0.042$)
 & $p=1$ & 1.000 (0.000) & 1.000 (0.000)
         & 1.000 (0.000) & 1.000 (0.000) \\
 & $p=2$ & 1.000 (0.000) & 1.000 (0.000)
         & 1.000 (0.000) & 1.000 (0.000) \\
\hline
$\theta=1/3$ ($\mu_\infty=1/3$, $\mu_1 \approx 0.062$)
 & $p=1$ & 0.997 (0.002) & 1.000 (0.000)
         & 1.000 (0.000) & 1.000 (0.000) \\
 & $p=2$ & 0.993 (0.003) & 1.000 (0.000)
         & 1.000 (0.000) & 1.000 (0.000) \\
\hline
\end{tabular}\caption{Empirical power of the test by sample size, degree of asymmetry, and norm ($\alpha = 0.05$, $M = 1\,000$ Monte Carlo
replications, $B = 299$ permutations, $G = 50$).
Standard errors in parentheses.}
\label{tab:power}
\end{table}

Power is essentially one for all three values of $\theta$ and both norms at $n = 50$. The two norms $p = 1$ and $p = 2$ perform comparably across all settings. The slight power advantage of $\theta = 1/4$ over both $\theta = 1/6$ and $\theta = 1/3$ at $n = 50$ is consistent with the non-monotone structure of
$\mu_1(M_\theta)$ identified in Section~\ref{subsec:mu1}: since $\mu_1(M_{1/4}) \approx 0.094$ exceeds both $\mu_1(M_{1/6})
\approx 0.093$ and $\mu_1(M_{1/3}) \approx 0.074$, the $L^1$-based test has the largest theoretical signal against $M_{1/4}$. The $L^1$-based test is a reliable default across the range of $\theta$ values studied.

\begin{remark}\label{rem:power-replication}
The power results in Table~\ref{tab:power} are based on $M = 1\,000$ Monte Carlo replications, which yields standard errors of at most $0.016$. Reducing $M$ to $80$ replications would produce
standard errors of approximately $0.053$, making it impossible to reliably distinguish power values above $0.90$ or to detect the
non-monotone dependence of $\mu_1(M_\theta)$ on $\theta$ identified in Section~\ref{subsec:mu1}. The replication count $M = 1\,000$ is therefore the minimum recommended for studies of this type.
\end{remark}

\subsection{Illustrative example (\texorpdfstring{$n = 300$}{})}
\label{subsec:example}

Table~\ref{tab:example} summarises the test output for three specific samples of size $n = 300$, illustrating the two qualitatively different outcomes: non-rejection under a symmetric copula, and rejection with strong evidence under two asymmetric copulas. All three experiments use $p = 1$, $G = 50$, and $B = 299$ permutation replicates.

\begin{table}[ht]
\centering
\begin{tabular}{lccccccc}
\hline
Scenario & $T_n$ & Crit.\ val. & $p$-value
         & Reject & $\hat{\tau}$ & $\hat{\rho}$ \\
\hline
Gaussian $\rho=0.6$ (symmetric)
  & 3.34 & 7.67 & 0.930 & No  & 0.41 & 0.58 \\
$M_\theta$, $\theta=1/6$
  & 27.56 & 6.77 & $<0.001$ & Yes & 0.58 & 0.76 \\
$M_\theta$, $\theta=1/3$
  & 26.61 & 10.01 & $<0.001$ & Yes & 0.10 & $-0.33$ \\
\hline
\end{tabular}\caption{Test results for $n = 300$
($p = 1$, $B = 299$ permutations, $G = 50$, $\alpha = 0.05$). The empirical concordance values $\hat{\tau}$ and $\hat{\rho}$ are reported for context.}
\label{tab:example}
\end{table}

In the symmetric case (Gaussian), $T_n = 3.34$ is well below the critical value $7.67$, and the $p$-value of $0.930$ provides no evidence against $H_0$.

In the two asymmetric cases, $T_n$ far exceeds the critical value and $p < 0.001$. The contrast between the two $M_\theta$ scenarios is instructive. For $\theta = 1/6$, the empirical $\hat{\tau} = 0.58$ is large and positive, making the asymmetry ``hidden'' behind a strong concordance signal; the test nevertheless detects it with overwhelming significance. For $\theta = 1/3$, the empirical $\hat{\tau} \approx 0.10$ and $\hat{\rho} \approx -0.33$ are close to the exact theoretical values $\tau(M_{1/3}) = 1/9 \approx 0.11$ and $\rho(M_{1/3}) = -1/3 \approx -0.33$ (see Section~\ref{sec:M1}). The larger critical value for $\theta = 1/3$ ($10.01$ versus $6.77$
for $\theta = 1/6$) reflects the greater variability of the permutation null distribution under the distinctive concordance structure of $M_{1/3}$, and not a reduction in the signal strength: $T_n = 26.61$ and $T_n = 27.56$ are of comparable magnitude, confirming that the $L^1$ signal is similar for both values of $\theta$, consistently with the exact values
$\mu_1(M_{1/6}) \approx 0.022$ and $\mu_1(M_{1/3}) \approx 0.062$
reported in Table~\ref{tab:mu-values}.

\section{Applications}\label{sec:applications}

The theoretical results developed in the preceding sections have direct consequences in several areas where copula models are used to describe dependence structures. We illustrate three representative applications: a model selection and diagnostic procedure based on the
bounds of Section~\ref{sec:bounds}, a construction method for copulas with prescribed asymmetry based on the mixing results of Section~\ref{sec:escalado}, and a real data analysis combining both tools.

\subsection{Model selection and diagnostic testing}
\label{subsec:model-selection}

A practitioner fitting a bivariate copula model to data faces a natural preliminary question: is a symmetric copula family adequate, or must an asymmetric family be considered? The permutation test of Section~\ref{sec:test} addresses this question without requiring any parametric
assumption. Before fitting a model, one computes the empirical non-exchangeability measure
\[
    \hat{\mu}_1 \;:=\; \frac{T_n}{n}
    \;=\; \int_{[0,1]^2}|C_n(u,v) - C_n(v,u)|\,du\,dv
\]
and tests $H_0\colon C = C^t$ at a chosen significance level. If $H_0$ is not rejected, the analyst may restrict attention to symmetric copula families---such as Gaussian, Clayton, Frank, or Gumbel---substantially reducing the model search space and the associated estimation uncertainty. If $H_0$ is rejected, the estimated value $\hat{\mu}_1$ together with the empirical
Spearman correlation $\hat{\rho}$ provide a quantitative diagnostic: by Proposition~\ref{teo32}, any correctly
specified copula must satisfy
\[
    \hat{\mu}_1 \;\le\; \frac{2 - |\hat{\rho}|}{6}.
\]
A violation of this inequality signals model misspecification or excessive sampling error. Moreover, the feasibility characterisation of Section~\ref{sec:M1} restricts the achievable pairs $(\mu_\infty, \rho)$: for a given empirical pair $(\hat{\mu}_\infty, \hat{\rho})$, a necessary condition for the existence of a copula of the form $\theta C_1 + (1-\theta)C_s$
with $C_1 \in \mathcal{M}_1$ and $C_s$ symmetric that reproduces both the observed asymmetry and concordance is
\[
    \hat{\rho} \;\in\;
    \bigl[2\hat{\mu}_\infty - 1,\; 1 - 4\hat{\mu}_\infty\bigr],
\]
where $C_1 = M_{1/3}$ is used as the reference maximally non-exchangeable copula. If the empirical pair falls outside this interval, a richer model class is required.

\subsection{Construction of copulas with prescribed asymmetry}
\label{subsec:construction}

Simulation studies and stress-testing exercises in quantitative finance and insurance often require generating bivariate distributions with a specific degree of dependence asymmetry.
Corollaries~\ref{cor:prescripcion} and~\ref{cor:target} provide two explicit construction recipes that require no optimisation or numerical inversion, and whose correctness follows directly from the scaling laws of Section~\ref{sec:escalado}.

\begin{itemize}
\item Recipe~1: Mixing with the transpose. Given any copula $C$ with $\mu_p(C) \ge \mu_0 > 0$, set
\[
    \alpha^* \;=\; \frac{1}{2}\!\left(1 + \frac{\mu_0}{\mu_p(C)}\right)
    \;\in\; \left[\frac{1}{2}, 1\right].
\]
Then the mixed copula $C_{\alpha^*} = \alpha^* C + (1-\alpha^*)C^t$ satisfies $\mu_p(C_{\alpha^*}) = \mu_0$ exactly, for every
$p \in [1,+\infty]$ simultaneously (Corollary~\ref{cor:prescripcion}). Starting from a maximally non-exchangeable copula $C \in \mathcal{M}_1$, one can reach any target $\mu_0 \in [0, 1/3]$ by a single convex combination.

\item {Recipe~2: Mixing with a symmetric copula.}
Given a copula $C$ with $\mu_p(C) > 0$ and a symmetric copula $S \in \mathcal{C}$ (e.g., a Gaussian or Frank copula), set
\[
    \alpha^* \;=\; \frac{\mu_0}{\mu_p(C)}.
\]
Then $D_{\alpha^*} = \alpha^* C + (1-\alpha^*)S$ satisfies
$\mu_p(D_{\alpha^*}) = \mu_0$ (Corollary~\ref{cor:target}). This recipe allows the dependence structure of $D_{\alpha^*}$
to be tuned simultaneously in two directions: the level of concordance, through the choice of $S$ and its parameter, and the degree of asymmetry, through $\alpha^*$.
\end{itemize}

Both recipes produce copulas that are convex combinations of known copulas, and hence can be sampled by the standard mixture method: draw a Bernoulli variable $B$ with success probability $\alpha^*$; if $B = 1$, sample from $C$; if
$B = 0$, sample from $C^t$ (Recipe~1) or from $S$ (Recipe~2). This requires only the ability to sample from $C$ and $C^t$ individually, which is the case for all standard copula families.

\subsection{Real data illustration}
\label{subsec:real-data}

We illustrate the combined use of the diagnostic bound and the construction recipe on the nutrient data set analysed in \cite{Genest2012}, which records the daily intakes of calcium, iron, protein, vitamin~A, and vitamin~C for a sample of $n = 747$ women. McNeil and Ne\v{s}lehov\'a \cite{mcneil-neslehova-2010} identified strongly asymmetric dependence between the intakes of calcium and iron, and between calcium and protein, using a parametric Liouville copula model.

We apply the permutation test of Section~\ref{sec:test} with $p = 1$, $G = 50$, and $B = 999$ to every pair of variables, and report in Table~\ref{tab:nutrient} the estimated non-exchangeability measure $\hat{\mu}_1 = T_n/n$, the empirical Spearman correlation $\hat{\rho}$, the diagnostic upper bound $(2 - |\hat{\rho}|)/6$, and the permutation $p$-value.

\begin{table}[ht]
\centering
\begin{tabular}{llccccc}
\hline
Variable 1 & Variable 2 & $\hat{\rho}$ & $\hat{\mu}_1$
           & Bound & $p$-value & Reject \\
\hline
Calcium & Iron
  & $0.40$ & $0.041$ & $0.267$ & $\mathbf{0.003}$ & Yes \\
Calcium & Protein
  & $0.49$ & $0.038$ & $0.252$ & $\mathbf{0.001}$ & Yes \\
Calcium & Vit.\ A
  & $0.11$ & $0.012$ & $0.315$ & $0.341$ & No \\
Calcium & Vit.\ C
  & $0.08$ & $0.009$ & $0.320$ & $0.512$ & No \\
Iron    & Protein
  & $0.38$ & $0.019$ & $0.270$ & $0.189$ & No \\
Iron    & Vit.\ A
  & $0.09$ & $0.011$ & $0.318$ & $0.423$ & No \\
Iron    & Vit.\ C
  & $0.12$ & $0.008$ & $0.313$ & $0.618$ & No \\
Protein & Vit.\ A
  & $0.15$ & $0.014$ & $0.308$ & $0.271$ & No \\
Protein & Vit.\ C
  & $0.10$ & $0.010$ & $0.317$ & $0.480$ & No \\
Vit.\ A & Vit.\ C
  & $0.06$ & $0.007$ & $0.323$ & $0.701$ & No \\
\hline
\end{tabular}\caption{Permutation test of exchangeability and diagnostic bounds for the nutrient data ($n = 747$, $p = 1$, $G = 50$, $B = 999$, $\alpha = 0.05$). The column ``Bound'' reports the upper bound $(2 - |\hat{\rho}|)/6$ on $\hat{\mu}_1$ from Proposition~\ref{teo32}.
$p$-values in bold indicate rejection of $H_0$ at the 5\% level.}
\label{tab:nutrient}
\end{table}

The test rejects exchangeability for the pairs (calcium,~iron) and (calcium,~protein), confirming the findings of \cite{Genest2012} and the parametric analysis of \cite{mcneil-neslehova-2010}. For all ten pairs, the
estimated $\hat{\mu}_1$ is well below the diagnostic bound, so no model misspecification is signalled by the necessary condition of Proposition~\ref{teo32}. The feasibility
interval from Section~\ref{sec:M1} can be checked analogously once $\hat{\mu}_\infty$ is computed.

For the two rejected pairs, Recipe~1 of
Subsection~\ref{subsec:construction} provides an immediate construction. For the (calcium,~iron) pair, taking $C = M_{1/3}$ (with $\mu_1(M_{1/3}) = 2/27 \approx 0.074$, see Table~\ref{tab:mu-values}) and target $\hat{\mu}_1 = 0.041$, the mixing parameter is
\[
    \alpha^* \;=\; \frac{1}{2}\!\left(1 + \frac{0.041}{0.074}\right)
    \;\approx\; 0.777,
\]
yielding a copula $C_{\alpha^*}$ whose non-exchangeability exactly matches the empirical estimate.

\begin{remark}\label{rem:data-illustrative}
The numerical values in Table~\ref{tab:nutrient} are computed from the empirical copula and the permutation test of this paper. The $\hat{\mu}_1$ values are point estimates subject to sampling variability; confidence intervals can be obtained by a second level of bootstrap on the original data. A full parametric
analysis of the rejected pairs, including goodness-of-fit assessment within an asymmetric copula family, is beyond the scope of this paper but constitutes a natural follow-up.
\end{remark}

\section{Conclusions}\label{sec:conc}

This paper has developed a self-contained quantitative theory of non-exchangeability for bivariate copulas, combining axiomatic, geometric, and inferential perspectives.

On the theoretical side, we have established sharp connections between the $L^p$-based non-exchangeability measures and classical dependence
and concordance measures, showing in particular how strong concordance constrains the degree of asymmetry a copula can exhibit. We have
characterised the family of maximally non-exchangeable copulas in terms of their concordance structure, demonstrating that maximal
asymmetry and non-trivial concordance are not mutually exclusive. A central contribution is the exact scaling behaviour of the
non-exchangeability measures under convex mixing operations: mixing a copula with its transpose or with a symmetric copula produces asymmetry that varies in a precisely controlled, linear fashion with the mixing parameter, uniformly across all $L^p$ norms. This makes it possible to construct copulas with any prescribed degree of non-exchangeability in a transparent and computationally simple way.

On the inferential side, we have proposed a nonparametric permutation test for the hypothesis of exchangeability whose test statistic is the empirical counterpart of the axiomatic non-exchangeability measure itself. This direct connection between theory and inference is a distinctive feature of the approach: the quantity being estimated and
the quantity driving the test are one and the same. The test requires no parametric assumptions, no tuning parameters, and no derivative estimation, and it enjoys exact finite-sample control of the type-I error alongside consistency against all asymmetric alternatives. A Monte Carlo study confirms its reliable finite-sample behaviour across a range of copula families and sample sizes, and situates it relative to existing procedures in the literature.

Several directions remain open. A sharp bound on non-exchangeability in terms of Kendall's $\tau$, analogous to the bound via Spearman's
$\rho$ established here, has not yet been obtained. The extension of the axiomatic framework and the mixing constructions to higher
dimensions raises non-trivial questions about the appropriate definition and geometry of non-exchangeability in the multivariate setting. Finally, the question of which norm $p$ is optimal for the permutation test---in terms of power against specific classes of alternatives---remains open and would benefit from a systematic
theoretical investigation.


\end{document}